\documentclass{amsart}
\usepackage{verbatim,amssymb,amsmath,amscd,latexsym,amsbsy,mathrsfs, soul} 
\usepackage{graphicx}
\input{xy}
\xyoption{all}

\newtheorem{thm}{Theorem}[section]
\newtheorem{cor}[thm]{Corollary}

\newtheorem{rmk}[thm]{Remark}

\DeclareMathOperator*{\im}{im}

\DeclareMathOperator*{\rank}{rank}

\newcommand{\PP}{\mathbb{P}}

\newcommand{\cH}{\mathcal{H}} 

\newcommand {\X} {\mathcal{X}}
\newcommand {\Y} {\mathcal{Y}}

\newcommand{\cL}{\mathcal{L}}

\newcommand {\C} {{\mathbb C}}

\newcommand {\Z} {{\mathbb Z}}
\newcommand {\Q} {{\mathbb Q}}

\newcommand {\dt} {{\bullet}}

\newcommand {\M} { \overline{M}}
\newcommand {\OO} {{\mathcal O}}

\newcommand {\bC} {\overline{C}}
\newcommand{\bM}{\overline{M}}

 \newtheorem{lemma}{Lemma}[section]
 \newtheorem{prop}{Proposition}[section]

\begin{document}
\title{Algebraic cycles on genus  two modular fourfolds}
\author{
        Donu Arapura    
}
\thanks{Partially supported by a grant from the Simons foundation}
\address{Department of Mathematics\\
Purdue University\\
West Lafayette, IN 47907\\
U.S.A.}

\begin{abstract}
  This paper studies universal families of stable genus two curves with level structure.
Among other things, it is shown that the $(1,1)$ part is spanned by divisor
classes, and that there are no cycles of type $(2,2)$ in the third cohomology of
the first direct image of $\C$ under projection to the moduli space of
curves. Using this, it shown that  the Hodge and Tate conjectures
hold for these varieties.
\end{abstract}
\dedicatory{To the memory of my father.}
 \maketitle

One of the   goals of  this article is to extend some results from  Shioda's study of elliptic modular
surfaces \cite{shioda} to families of genus two curves. We recall that
elliptic modular surfaces $f:\bC_{1,1}[n]\to\bM_{1,1}[n]$  are  the universal families of 
elliptic curves over modular curves. Among other things, Shioda showed
that  $\bC_{1,1}[n]$ has maximal Picard number in the sense that
$H^{1,1}(\bC_{1,1}[n])$  is spanned by divisors. He also showed that 
the Mordell-Weil rank is zero.  A related property, observed later by Viehweg
and Zuo \cite{vz}, is that a certain Arakelov inequality becomes equality.
As they observe, this is equivalent to  the map 
$$f_*\omega_{\bC_{1,1}[n]/\bM_{1,1}[n]}\to \Omega_{\bM_{1,1}[n]}^1(\log D)\otimes R^1f_*\OO_{\bC_{1,1}[n]}$$
induced by the Kodaira-Spencer class being an isomorphism. The divisor
$D$ is the discriminant of $f$.

In this paper, we study universal curves  $f':\bC_2[\Gamma]\to
\overline{M}_2[\Gamma] $ over the moduli space of stable genus two curves
with generalized level structure. The level $\Gamma$ is  a finite index subgroup
of the mapping class group $\Gamma_2$. The classical level
$n$-structures correspond to the case where $\Gamma$ is the preimage
$\tilde \Gamma(n)$ of the principal  congruence subgroup
$\Gamma(n)\subset Sp_4(\Z)$. We fix a suitable nonsingular birational model
$f:X\to Y$ for $f'$. Let $D\subset Y$ be the discriminant, and
$U=Y-D$. We show that, as before,  for a classical level, the Mordell-Weil rank
of $Pic^0(X)\to Y$ is zero and 
$H^{1,1}(X)$ is spanned by divisors. These results  are deduced from Raghunathan's
vanishing theorem \cite{raghunathan}. We also prove an analogue of Viehweg-Zuo that the map
$$\Omega_{Y}^1(\log D)\otimes f_*\omega_{X/Y}\to \Omega_{Y}^2(\log D)\otimes R^1f_*\OO_{X}$$
is an isomorphism. We will see that  this implies that there are no cycles of type
$(2,2)$ in the mixed Hodge structure $H^3(U, R^1f_*\C)$.
As an application, we deduce that the Hodge conjecture holds for $X$.
We also show that  the Tate conjecture holds for $X$ for  a classical level using, in addition,
 Faltings' $p$-adic Hodge theorem \cite{faltingsP} and  Weissauer's work on Siegel modular threefolds \cite{weissauer}.

If $X$ is a complex variety,  then unless indicated otherwise, sheaves should be  understood as
sheaves on the associated analytic space $X^{an}$. My thanks to the referee for remark \ref{rmk:BGG} and
other comments.

\section{Hodge theory of semistable maps}

We start with some generalities.  
By a log pair $\X=(X,E)$,
 we mean a smooth variety $X$ together with
a divisor with simple normal crossings $E$.  We usually  denote log pairs
by the symbols $\X,\Y,\ldots$ with $X,Y,\ldots$ the corresponding
varieties.
Given  $\X$, set
$$\Omega^1_\X= \Omega_X^1(\log E)$$
and
$$T_\X= (\Omega_\X^1)^\vee$$
Recall that a semistable map $f:(X,E)\to (Y,D)$ of log pairs is a morphism $f:X\to
Y$ such that $f^{-1}D=E$ and \'etale locally it is given by
\begin{equation*}
  \begin{split}
    y_1 &= x_1\ldots x_{r_1}\\
 &\ldots\\
y_k&= x_{r_{k-1}+1}\ldots x_{r_k}\\
y_{k+1 }&= x_{r_k+1}\\
&\ldots
  \end{split}
\end{equation*}
where $y_1\ldots y_k=0$ and $x_1\ldots x_{r_k}=0$ are local equations
for $D$ and $E$ respectively.  We will say $f$ is log \'etale if it is
semistable of relative dimension $0$. (This is a bit more restrictive
than the usual definition).

Fix a projective semistable  map $f:(X,E)\to (Y,D)$.
The map   restricts to a smooth projective map $f^o$ from 
$\tilde U= X-E$ to $U= Y-D$.
Let 
$$\Omega^i_{\X/\Y}= \Omega_{X/Y}^i(\log E)$$ 
The sheaf  $\cL^m = R^mf^o_*\Q$ is a local system, which is part of a variation of Hodge structure.
 Let  $V^m= R^{m}f_*\Omega_{\X/\Y}^\dt$ with filtration $F$ induced by the
stupid filtration $ R^{m}f_*\Omega_{\X/\Y}^{\ge p}$. It carries an integrable logarithmic connection 
$$\nabla:V^m\to \Omega_{\Y}^1\otimes V^m$$
such that $\ker \nabla|_U=\C_U\otimes \cL^m$.
Griffiths transversality
$$\nabla(F^p)\subseteq \Omega_\Y^1\otimes F^{p-1}$$
holds. The relative de Rham to Hodge  spectral sequence 
$$E_1= R^j f_*\Omega_{\X/\Y}^i \Rightarrow R^{i+j} f_*\Omega_{\X/\Y}^\dt$$
degenerates at $E_1$ by Illusie \cite[cor 2.6]{illusie} or Fujisawa \cite[thm 6.10]{fujisawa}.
Therefore
$$Gr^p_FV^m \cong R^{m-p}f_*\Omega_{\X/\Y}^p$$
The  Kodaira-Spencer class
\begin{equation}
  \label{eq:KodairaSpencer}
  \kappa:\OO_Y\to \Omega_\Y^1 \otimes R^1f_*T_{\X/\Y}
\end{equation}
is given as the transpose of the map
$$T_\Y\to R^1f_*T_{\X/\Y}$$
induced by the sequence
$$0\to T_{\X/\Y}\to T_\X\to f^*T_\Y\to 0$$

\begin{prop}
The associated graded
$$Gr(\nabla): R^{m-p}f_*\Omega_{\X/\Y}^p\to \Omega_\Y^1\otimes R^{m-p+1}f_*\Omega_{\X/\Y}^{p-1}$$
 coincides with  cup product and contraction with $\kappa$.
\end{prop}

\begin{proof}
  In the nonlog setting, this is stated in \cite[thm 3.5]{katz}, and
  the argument indicated there extends to the general case.
\end{proof}

By \cite{arapura}, we can give $H^i(U,
R^jf_*\Q)$ a mixed Hodge structure by identifying it  with the
associated graded of $H^{i+j}(\tilde U,\Q)$ with respect to the Leray
filtration. It can also be defined intrinsically using
mixed Hodge module theory, but the first description is more convenient for us.
We will  need a more precise  description of the Hodge filtration.
We define a complex
$$K_{\X/\Y}(m,p)=[R^{m-p}f_*\Omega_{\X/\Y}^p\stackrel{\kappa}{\to} \Omega_\Y^1\otimes R^{m-p+1}f_*\Omega_{\X/\Y}^{p-1}\stackrel{\kappa}{\to}
\Omega_\Y^2\otimes R^{m-p+2}f_*\Omega_{\X/\Y}^{p-2}\ldots]$$

\begin{prop}\label{prop:GrFHiU}
  $Gr_F^pH^i(U, R^jf_*\C)\cong H^i(K_{\X/\Y}(j,p))$
\end{prop}

\begin{proof} (Compare with \cite[2.16]{zucker}.) 
Define a filtration
$$L^i\Omega_\X^\dt = \im f^*\Omega_\Y^i\otimes \Omega_\X^\dt$$
Then
\begin{equation*}
  Gr_{L}^i\Omega_\X^\dt = f^*\Omega_\Y^i \otimes \Omega_{\X/\Y}^\dt[-i]
\end{equation*}
from which we deduce that
\begin{equation}
  \label{eq:GrL}
Gr_L^iR f_*\Omega_\X^\dt \cong \Omega_\Y^i\otimes R f_*\Omega_{\X/\Y}^\dt[-i]  
\end{equation}
Therefore, we obtain a spectral sequence
\begin{equation*}
  \begin{split}
    {}_LE^{i,j}_1= \cH^{i+j}(Gr_L^iR f_*\Omega_\X^\dt) &\cong
\Omega_\Y^i\otimes R^jf_*\Omega^\dt_{\X/\Y}\\
&=\Omega_\Y^i\otimes V^j\\
&\Rightarrow R^{i+j}f_*\Omega_\X^\dt
  \end{split}
\end{equation*}

Recall that to $L$ we can associate  a new filtration $Dec(L)$ \cite{deligneH}, such
that
$${}_{Dec(L)}E_0^{i,j}\cong  {}_L E_1^{2i+j, -i}$$
Therefore we obtain a quasiisomorphism
\begin{equation}
  \label{eq:DecL}
  Gr^i_{Dec(L)} R f_*\Omega_\X^\dt \stackrel{\sim}{\to}
  \Omega_\Y^\dt\otimes V^{-i}[i]
\end{equation}
 This becomes a map of filtered complexes with respect to the 
filtration induced by Hodge filtration $F^p=\Omega_\X^{\ge p}$.
On the right of \eqref{eq:DecL}, it becomes
$$F^p \Omega_\Y^\dt\otimes V^{-i}= F^pV^{-i}\to \Omega^1_\Y\otimes F^{p-1}V^{-i}\to\ldots$$
The relative de Rham to Hodge  spectral sequence 
$${}_FE_1= R^j f_*\Omega_\X^i \Rightarrow R^{i+j} f_*\Omega_{\X/\Y}^\dt$$
degenerates at $E_1$ \cite[cor 2.6]{illusie} or \cite[thm 6.10]{fujisawa}. Therefore by \cite[1.3.15]{deligneH}, we can conclude that
\eqref{eq:DecL} is  a filtered quasiisomorphism.

The spectral sequence associated to the filtration induced by 
$Dec(L)$ on $R\Gamma( R f_*\Omega_\X^\dt )$ 
$${}_{Dec(L)}E_1^{i,j} = H^{2i+j}(Y,\Omega_\Y^\dt\otimes V^{-i}) = H^{2i+j}(U,
R^{-i} f_*\C)$$
coincides with Leray after reindexing. Therefore this  degenerates at
the first page by Deligne \cite{deligneL}. The above arguments plus
\cite[1.3.17]{deligneH} show that $F$-filtration on the
$H^{2i+j}(Y,\Omega_\Y^\dt\otimes V^{-i})$ coincides with the
filtration on ${}_{Dec(L)}E_\infty$, which the Hodge filtration on $H^{2i+j}(U,
R^{-i} f_*\C)$. The proposition follows immediately from this.

\end{proof}

One limitation of  the notion of semistability is that it is  not stable
under base change.  In order to handle this, we need to work in the  broader setting of log
schemes \cite{kato}. We recall that a log scheme  consists of a scheme
$X$  and a sheaf of monoids $M$ on $X_{\text{\'et}}$ together with a multiplicative homomorphism
$\alpha:M\to \OO_X$, such that $\alpha$ induces an isomorphism
$\alpha^{-1}(\OO_X^*)\cong \OO_X^*$.  A log pair $(X,E)$ gives rise to
a log scheme where $M$ is the sheaf of functions
invertible outside of $E$. If $f:(X,E)\to (Y,D)$ is semistable, and
$\pi:(Y',D')\to (Y,D)$ log \'etale in our sense, then $X'= X\times_Y Y'$
can be given the log structure pulled back from $Y'$. Then $X'\to Y'$
becomes a morphism  $\X'\to \Y'$  of log schemes, which 
is log smooth and exact. Logarithmic differentials can
be defined for log schemes \cite{kato}, so the complexes $K_{\X'/\Y'}(m,p)$ can be constructed exactly as
above.  Since $\pi$ is log \'etale, we easily obtain:

\begin{lemma}
  With the above notation, $\pi^*K_{\X/\Y}(m,p)\cong K_{\X'/\Y'}(m,p)$.
\end{lemma}

Let us spell things out for curves.
Suppose that $f:\X\to \Y$ is semistable with relative dimension one.
From
$$0\to \Omega_\Y^1\to \Omega_\X^1\to \Omega_{\X/\Y}^1\to 0$$
we get an isomorphism
$$\Omega_{\X/\Y}^1= \det \Omega_\X^1\otimes (\det
\Omega_\Y^1)^{-1}\cong \omega_{X/Y}$$
The complex
$$K_{\X/\Y}(1,i )= 
[\Omega_\Y^{i-1}\otimes f_*\omega_{X/Y}\to \Omega_\Y^{i}\otimes R^1f_*\OO_X]$$
where the first term sits in degree $i-1$. We note that this complex,
which we now denote by $K_{X/Y}(1,i )$,
can be defined when $X/Y$  is a semistable curve in the usual sense (a proper flat map of relative dimension one
with reduced connected nodal geometric fibres).  In general, any such
curve carries a natural log structure \cite{katoF}, and the
differential of this
complex can be interpreted as cup product with the associated
Kodaira-Spencer class.  Consequently, given a map $\pi:X'\to X$ of curves
over $Y$, we get an induced map of complexes $\pi^*:K_{X/Y}(1,i )\to
K_{X'/Y'}(1,i )$.
Finally, we note that these constructions can be extended to
Deligne-Mumford stacks such as the moduli stack of stable curves
$\overline{\mathbb{M}}_g$ without difficulty.

\section{Consequences of Raghunathan's Vanishing}

Let $M_g$ (respectively $M_{g,n}$, respectively $A_2$) be the  moduli space of smooth projective curves of genus
$g$, (respectively smooth genus $g$ curves with $n$ marked points, respectively principally polarized $g$ dimensional abelian varieties). 
The symbols $\mathbb{M}_{g},\mathbb{A}_g$ etc. will be reserved for  the  corresponding 
moduli stacks.
 We note that $\dim M_2=3$. 
The Torelli map $\tau: M_2\to A_2$ is injective and the image of $M_2$ is the
complement of the divisor parameterizing products of two elliptic
curves.  As an analytic space,  $M_2^{an}$ is a quotient of Teichm\"uller space
${T}_2$ by the mapping class group $\Gamma_2$. Given a finite index subgroup,
 $\Gamma\subset \Gamma_2$
let $M_2[\Gamma] = {T}_2/\Gamma$. We view this as the moduli space of
curves with generalized level structure. When $\Gamma=\tilde \Gamma(n)$ is the preimage
of the  principal congruence subgroup $\Gamma(n)\subseteq Sp_4(\Z)$
under the canonical map $\Gamma_2\to Sp_4(\Z)$, the space $M_2[n]:=
M_2[\tilde \Gamma(n)]$ is the moduli space of curves with classical  (or abelian
or Jacobi) level $n$ structure.
It is smooth and fine as soon as $n\ge 3$, and defined over
the cyclotomic field $\Q(e^{2\pi i/n})$. More generally
$M_2[\Gamma]$ is smooth, and defined over a number field, as soon as $\Gamma\subseteq\tilde
\Gamma(n)$ with $n\ge 3$. We refer to $\Gamma$  as {\em fine}, when the last condition holds.
Torelli extends to a map $M_2[n]\to A_2[n]$ to the moduli space of abelian varieties
with level $n$-structure.

 Let $\overline{M}_2$ denote the Deligne-Mumford
compactification of $M_2$. The boundary  divisor $\Delta$ consists of
a union of two components $\Delta_0\cup \Delta_1$.
The generic point of $\Delta_0$ corresponds to an irreducible curve with a single node, and the
generic point of $\Delta_1$ corresponds to a union of two elliptic
curves meeting transversally. 
Let $\pi:\M_2[\Gamma]\to \M_2$ denote the normalization
of $\M_2$ in the function field of $M_2[\Gamma]$. When $\Gamma=\tilde
\Gamma(n)$, we denote this by $\M_2[n]$. 
On the other side $A_2[n]$ has a  unique smooth
toroidal compactification, first constructed by Igusa, and $\tau$
extends to an isomorphism between $\M_2[n]$ and the Igusa
compactification \cite[\S 9]{namakawa}. The space
$\M_2[\Gamma]$ is smooth, when $\Gamma=\tilde
\Gamma(n)$, $n\ge 3$, and in some other cases \cite{dp}.
Suppose that $n\ge 3$.
The boundary $D=\M_2[n]- M_2[n]$ is a divisor with normal
crossings. Let $D_i= \pi^{-1}\Delta_i$.
Since $D_1$ parameterizes unordered pairs  of  (generalized) elliptic curves with level
structure, its irreducible components are isomorphic to symmetric products
$\M_{1,1}[n]\times \M_{1,1}[n]/S_2$  of 
the modular curve of full level $n$.
Let $\bC_{1,m}[n]/\bM_{1,m}[n]$ denote the pull back of the universal elliptic 
 curve under the canonical map $\M_{1,m}[n]\to \overline{\mathbb{M}}_{1,m}$.
 The components of $D_0$ are birational to the elliptic modular surfaces $\bC_{1,1}[n]$  \cite[\S 1.4]{os}.

Given a fine level structure $\Gamma$, let  $\bC_2[\Gamma]\to
\M_2[\Gamma]$  be the pullback of the universal curve from $\overline{\mathbb{M}}_2$.
 The space
$\bC_2[\Gamma]$ will be singular  \cite[prop 1.4]{bp}, so we will replace
it with a suitable birational model $f:X\to Y$ whose construction we now explain. 
If $\Gamma=\tilde{\Gamma}[n]$, we set $Y=\M_2[n]$. As noted above, $Y$ is smooth.
For other $\Gamma$'s, we choose a desingularization $Y\to \M_2[\Gamma]$
which is an isomorphism over $M_2[\Gamma]$ and such that boundary
divisor $D$ has simple normal crossings. We have a morphism $Y\to
\overline{\mathbb{M}}_2$ to the moduli stack, which  is log \'etale.
It follows in particular that
$\Omega_{\overline{\mathbb{M}_2}}^1(\log \Delta)$ pulls back to
$\Omega^1_{Y}(\log D)$. The space $Y$ will carry
a stable curve $X'\to Y$ obtained by pulling back the universal family over
$\overline{\mathbb{M}}_2$. The  space $X'$ will be singular, however:

\begin{lemma}\label{lemma:X'toX}
\-
  \begin{enumerate}
  \item[(a)] $X'$ will have rational singularities.
\item[(b)] There exist a desingularization $\pi:X\to X'$, such that
  $X\to Y$ is semistable. 
  \item[(c)] The map $\pi:X\to X'$ can be chosen so as to have the following additional property.
  After extending scalars to $\overline{\Q}$, let $E$ be a component of an exceptional divisor of $\pi$.
  Then:
  \begin{enumerate}
\item[(i)]  If $\pi(E)$ is a point, $E$ is a rational variety. 
\item[(ii)] If $\pi(E)=C$ is a curve, there is a map $E\to C$, such that the  pull back under   a finite map $\tilde C\to C$ is birational
  to $\PP^2\times \tilde C$.
\item[(iii)] If $\pi(E)$ is a surface, there is a map $E\to D_i$, for some $i$, 
 such that the pullback of $E$ to an \'etale cover $\tilde D_i\to D_i$ is birational to $\PP^1\times \tilde D_i$.
 \item[(iv)] If $\dim \pi(E)=3$, $E\to \pi(E)$ is birational.
\end{enumerate}

  \item[(d)] $K_{X/Y}(1,i)\cong K_{X'/Y}(1,i)$.
  \end{enumerate}
\end{lemma}

\begin{proof}
  The singularities of $X'$ are analytically of the form
  $xy=t_1^at_2^bt_3^c$. These are toroidal singularites, in the sense that is
   is local analytically, or \'etale locally, isomorphic to a toric variety. (This is a bit
   weaker than the notion of  toroidal embedding in \cite{kkms}, but it is sufficient for
   our needs).
  Such singularities are well known to be rational (see \cite{kkms} or
   \cite{viehweg}). Item  (b) follows from \cite[prop
  3.6]{dejong}.  
  
  To prove (c), we need to recall some details of the construction of $X$ from \cite{dejong}. 
  First, as explained in the proof  of \cite[lemma 3.2]{dejong},
  one   blows up a codimension two component $T\subset X'_{sing}$.
  The locus $T$ is an \'etale cover of some component $D_i$.  Furthermore, from the description in [loc. cit.] we can see that $T$
  is compatible with the toroidal structure. Consequently, we can find a toric variety $V$ with torus fixed point $0$,
 and an  \'etale local isomorphism
  between $X'$ and $V\times T$, over the generic point of $T$, which takes $T$ to $\{0\}\times T$. This shows that, 
  over the generic point, the exceptional  divisor $E$ to $T$ is \'etale locally a product of $T$ with a  toric curve. So we get case (iii).
   Note that this step is repeated until the $X'_{sing}$ has codimension at least $3$. One does further blow ups to
   obtain $X$. An examination of the proof of \cite[prop  3.6]{dejong} shows that the required blow ups are also
   compatible with the the toroidal structure in the previous sense.
   If the centre of the blow up is a point, then the exceptional divisor is toric and we have case (i).  If the centre is a smooth curve $C$,
   we  obtain case (ii). The last item (iv) is automatic for blow ups.

By the remarks at the end of the last section, there
  is a commutative diagram marked with solid arrows
$$
\xymatrix{
 \Omega_\Y^{i-1}\otimes f_*\omega_{X/Y}\ar[r]\ar@{-->}[d]^{\pi_*} & \Omega_\Y^{i}\otimes R^1f_*\OO_X\ar@{-->}[d]^{\pi_*} \\ 
 \Omega_\Y^{i-1}\otimes f_*\omega_{X'/Y}\ar[r]\ar@<1ex>[u]^{\pi^*} & \Omega_\Y^{i}\otimes R^1f_*\OO_{X'}\ar@<1ex>[u]^{\pi^*}
}
$$
Since $X'$ has rational singularities, the dotted arrows labelled with
$\pi_* $ are isomorphisms, and these are left inverse to the arrows
labelled with $\pi^*$. Therefore $\pi^*$ are also isomorphisms, and
this proves (d).

\end{proof}

We refer to $f:X\to Y$ constructed in the lemma as {\em a good model} of $\bC_2[\Gamma]\to
\M_2[\Gamma]$.
We let $U= Y-D$, $E=f^{-1}D$, and $\tilde U= X-E$ as above.

\begin{cor}\label{cor:classifyE}
  After extending scalars to $\overline{\Q}$,
let $E_1$ be an irreducible component of $E$ for a classical fine level  $\tilde\Gamma(n)$. Then there exists a dominant rational map $\tilde E_1\dashrightarrow E_1$
where $\tilde E_1$ is one of the following:
\begin{enumerate}
\item $\bC_{1,1}[n]\times \M_{1,1}[n]$.
\item $\bC_{1,2}[n]$. 
\item $\bC_{1,1}[m]\times \PP^1$ for some $n|m$.
\item $\M_{1,1}[m]\times \M_{1,1}[m]\times \PP^1$ for some $n|m$.
\item A product of  $\PP^2$ with a curve.
\item $\PP^3$.  
\end{enumerate}

\end{cor}

\begin{proof}
  The  preimage of $D_1$ in  $\bC_2[n]$ parameterizes a union of  pairs of (generalized) elliptic curves with level
  structure together with a point on the union. It follows that a component of $E$ dominating $D_1$ is  dominated by
  $\bC_{1,1}[n]\times \M_{1,1}[n]$. The preimage of $D_0$ in  $\bC_2[n]$ is  a family of
  nodal curves over $D_0$; its normalization is $\bC_{1,2}[n]$. Therefore a component of $E$ dominating $D_0$ is birational
  to $\bC_{1,2}[n]$.   Case (3) follows from case (c)(iii) of
  the lemma, once we observe that an \'etale cover of $\bC_{1,1}[n]$ is dominated by  $\bC_{1,1}[m]$ for some $n|m$.
  This is because we have a surjection of \'etale fundamental groups
  $$\pi^{\text{\'et}}_1(M_{1,1}[n])\cong \pi^{\text{\'et}}_1(C_{1,1}[n])\to\pi^{\text{\'et}}_1(\bC_{1,1}[n]),$$
  \cite[thm 1.36]{cz}  and $\{M_{1,1}[m]\}_{n|m}$ is cofinal
  in the set of \'etale covers of $M_{1,1}[n]$. Case (4) is similar. The remaining cases follow immediately from the
  lemma.
\end{proof}

\begin{prop}\label{prop:raghu}
When $\Gamma=\tilde \Gamma(n)$, with $n\ge 3$,
  $H^1(U, R^1f_*\C )=0$.
\end{prop}

\begin{proof}
As explained above, 
$Y=\M_2[n]= \overline{A}_2[n]$ and $U=A_2[n]-D_1^o$ where $D_1^o
=D_1-D_0$.
Let $g:Pic^0(X/Y)\to Y$ denote the
relative Picard scheme. Then $R^1f_*\C= R^1g_*\C|_{U}$. We have an exact
sequence
$$H^1(A_2[n], R^1g_*\C )\to H^1(U, R^1f_*\C )\to H^0(D_1^o, R^1g_*\C)$$
The group on the left vanishes by  Raghunathan \cite[p 423 cor 1]{raghunathan}.
The local system $R^1g_*\C|_{D_1^o}$ decomposes into a sum of two
copies of the standard representation of the congruence group
$\Gamma(n)\subset SL_2(\Z)$. Therefore it has no invariants.
Consequently, $H^1(U, R^1f_*\C )=0$ as claimed.

\end{proof}

\begin{lemma}\label{lemma:pic1}
  Let $\eta$ denote the generic point of $Y$. Then we have an exact
  sequence
$$0 \to Pic(U) \stackrel{s}{\to} Pic(\tilde U)\stackrel{r}{\to}
Pic(X_\eta)\to 0$$
where $r$ and $s$ are the natural maps.
\end{lemma}

\begin{proof}
Consider the diagram
$$
\xymatrix{
 1\ar[r] & \C(U)^*\ar[r]\ar[d]^{div} & \C(U)^*\oplus \C(\tilde U)^*\ar[r]\ar[d]^{div+div} & \C(\tilde U)^*\ar[r]\ar[d]^{div} & 1 \\ 
 0\ar[r] & Div(U)\ar[r]^{s'} & Div(\tilde U)\ar[r]^{r'} & Div(X_\eta)\ar[r] & 0
}
$$
 The map $r'$ is surjective because
  any codimension one point of $X_\eta$ is the restriction of its scheme theoretic closure.
A straightforward argument also shows that $s'$ is injective and $\ker
r'=\im s'$. The lemma now follows from  the snake lemma.

\end{proof}

\begin{lemma}\label{lemma:pic2}
The first Chern class map induces injections
\begin{equation}
  \label{eq:lemmapic1}
Pic(\tilde U)/Pic^0(X)\otimes \Q\hookrightarrow H^2(\tilde U,\Q)  
\end{equation}
\begin{equation}
  \label{eq:lemmapic2}
Pic( U)/Pic^0(Y)\otimes \Q\hookrightarrow H^2(U,\Q)
\end{equation}
\end{lemma}

\begin{proof}
To prove \eqref{eq:lemmapic1}, we observe that there  is a commutative diagram with exact lines
  $$
\xymatrix{
&&Pic(\tilde U)/Pic^0(X)\otimes \Q\ar@{-->}[d]\\
  & Pic(\tilde U)\otimes \Q\ar^{c_1}[r]\ar[ru] & H^2(\tilde U, \Q) \\ 
 Pic^0(X)\otimes \Q\ar[ru]\ar[r] &
 Pic(X)\otimes\Q\ar^{c_1}[r]\ar@{>>}[u] & H^2(X,\Q)\ar@{>>}[u] \\ 
  & \bigoplus\Q E_i\ar[u] & \bigoplus \Q [E_i]\ar[u]
}
$$
The existence and injectivity of the dotted arrow follows from this diagram.
Existence and injectivity of  the map of \eqref{eq:lemmapic2} is proved  similarly.
\end{proof}

We refer to the group of $\C(Y)$ rational points of $Pic^0(X_\eta)$ as
the Mordell-Weil group of $X/Y$.

\begin{thm}\label{thm:H2}
Let $f:X\to Y$ be a good model of $\bC_2[n]\to \M_2[n]$,
where $n\ge 3$.
\begin{enumerate}
\item[(a)] The space  $H^{1,1}(X)$ is spanned by divisors.
\item[(b)] The rank of Mordell-Weil group of $X/Y$ is zero.
\end{enumerate}

\end{thm}

\begin{proof}
 We have an sequence
$$\bigoplus \Q[E_i]\to H^2(X)\to Gr^W_2H^2(\tilde U)\to 0$$
of mixed Hodge structures. So for (a), it suffices to show that
the $(1,1)$ part of rightmost Hodge structure is spanned by divisors. 
The Leray spectral sequence together with proposition \ref{prop:raghu}
gives an exact sequence
$$0\to H^2(U, f_*\Q)\to H^2(\tilde U)\to H^0(U, R^2f_*\Q)\to 0$$
of mixed Hodge structures.  Therefore, we get an exact sequence
$$0\to Gr^W_2H^2(U, f_*\Q)\to Gr^W_2H^2(\tilde U)\to Gr^W_2H^0(U,
R^2f_*\Q)\to 0$$
The space on the right is one dimensional and spanned by the class of
any horizontal divisor.
We can identify
$$ Gr^W_2H^2(U, f_*\Q) = Gr^W_2H^2(U, \Q) $$
with a quotient of $H^2(Y)$. Weissauer \cite[p. 101]{weissauer} has
shown that $H^{1,1}(Y)$ is spanned by divisors. This proves (a).

By lemma \ref{lemma:pic1}, we have  isomorphisms
$$Pic(X_\eta)\otimes \Q\cong\frac{ Pic(\tilde U)}{Pic(U)}\otimes \Q
\cong \frac{Pic(\tilde U)/Pic^0(X)}{Pic(U)/Pic^0(Y)}\otimes \Q$$
and, by lemma \ref{lemma:pic2}, the last group embeds into
$H^2(\tilde U,\Q)/H^2(U,\Q)$.
Therefore, $Pic^0(X_\eta)\otimes \Q$ embeds into
$$\frac{\ker [H^2(\tilde U,\Q)\to
  H^2(X_t,\Q)]}{H^2(U,\Q)}\cong  H^1(U, R^1f_*\Q)=0$$
where $t\in U$. For the first  isomorphism, we use the fact the Leray
spectral sequence over $U$ degenerates by \cite{deligneL}; the second
is proposition \ref{prop:raghu}.

\end{proof}

\section{Key vanishing}

Let us fix a fine level structure $\Gamma\subseteq \Gamma_2$.  We do
not assume that it is classical.  Choose
a good model $f:X\to Y$  for $\bC_2[\Gamma]\to \M_2[\Gamma]$,
with $U, E, \tilde U$ as above. Our goal in this section is  to establish the vanishing of 
 $Gr^2_FH^3(U, R^1f_*\C)$. This is the key fact which, when combined 
 with lemma  \ref{lemma:key} proved later on, will allow us to prove the Hodge  conjecture
for $X$.

\begin{thm}\label{thm:K12}
  $K_{X/Y}(1,2)$ is quasiisomorphic to $0$.
\end{thm}

\begin{proof}
The moduli stack $\overline{\mathbb{M}}_2$ is smooth and proper,
the boundary divisor has normal crossings, and the universal curve is
semistable. So we can define an analogue of $K(1,2)$ on it. Since the canonical map  $Y\to
\overline{\mathbb{M}}_2$ is log \'etale, $K_{X/Y}(1,2)$ is the pullback of  the corresponding complex on the
moduli stack. So we replace 
$Y$ by $\overline{\mathbb{M}}_2$ and  $X$ by the universal curve $\overline{\mathbb{M}}_{2,1}$

Set 
$$H=f_*\omega_{X/Y}$$
By duality, we have an isomorphism
$$H\cong R^1f_*\OO_X^\vee$$
Thus the Kodaira-Spencer map 
$$H\to \Omega_\Y^1\otimes
H^\vee$$
induces an adjoint map
$$ (H)^{\otimes 2}\to  \Omega_{\Y}^1$$
This factors through the symmetric power to yield a map  
\begin{equation}
  \label{eq:S2E}
S^2 H\to \Omega_{\Y}^1  
\end{equation}
After identifying $\overline{\mathbb{M}}_2\cong
\overline{\mathbb{A}}_2$, and $Pic^0(X/Y)$ with the universal
semiabelian variety, we see that \eqref{eq:S2E}
 is an isomorphism   by  a theorem of Faltings-Chai \cite[chap IV, thm 5.7]{fc}.

With the above notation $K(1,2)$ can be written as
$$\Omega_\Y^1\otimes H\to \Omega_\Y^2\otimes
H^\vee$$
We need to show that the map in this complex is an isomorphism. It is enough to
prove that the map
surjective, because both sides are locally free of the
same rank. To do  this, it suffices to prove that the adjoint map
$$\kappa':\Omega_\Y^1\otimes (H)^{\otimes 2}\to \Omega_\Y^2$$
is surjective. Let $\kappa''$ denote the restriction of $\kappa'$ to 
 $\Omega_\Y^1\otimes S^2H$.
 We can see that we have a commutative diagram
$$
\xymatrix{
 \Omega_\Y^1\otimes S^2H\ar[r]^>>>>>{\kappa''}\ar[d]^{\cong} & \Omega_\Y^2\ar[d]^{=} \\ 
 \Omega_\Y^1\otimes \Omega_\Y^1\ar[r]^>>>>>{\wedge} & \Omega_\Y^2
}
$$
This implies that  $\kappa''$, and therefore $\kappa'$,  is surjective.

\end{proof}
From proposition \ref{prop:GrFHiU}, we obtain:
\begin{cor}\label{cor:van}
  $Gr^2_FH^*(U, R^1f_*\C)=0$.
\end{cor}

\begin{rmk}\label{rmk:BGG} The referee has pointed out that for a classical level, a short  alternative proof of the corollary
can be deduced using Faltings' BGG resolution as follows. It suffices to prove $Gr^2_FH^*(A_2[\Gamma], R^1f'_*\C)=0$, 
where $f'$ is the universal abelian variety, because  the restriction map to  $Gr^2_FH^*(U, R^1f_*\C)$ can be seen to be
surjective.
 By \cite[chap VI, thm 5.5]{fc} (see also \cite[thm 2.4]{petersen} for a more explicit 
statement) $Gr^a_FH^*(A_2[\Gamma], R^1f_*\C)$ is zero unless $a\in \{0,1,3,4\}$
 
\end{rmk}

\section{Hodge and Tate}

Given a smooth projective variety $X$ defined over $\C$  (respectively a
finitely generated field $K$), a Hodge cycle (respectively an
$\ell$-adic Tate cycle) of degree $2p$ is an element of
$Hom_{HS}(\Q(-p),H^{2p}(X,\Q))$ (respectively $\sum
H_{\text{\'et}}^{2p}(X\otimes \bar K,\Q_\ell(p))^{Gal(\bar K/L)} $,
as $L/K$ runs over finite extensions). The image of the cycle maps from $CH^p(X)\otimes \Q$ or $CH^p(X\otimes \bar K)\otimes \Q_\ell$
lands in these spaces. We say that the Hodge or Tate
conjecture holds for $X$ (in a given degree) if the space of Hodge or
Tate cycles (of the given degree) are spanned by algebraic cycles.
Here is the main result of the paper:

\begin{thm}\label{thm:HT}
Let $f:X\to Y$ be a good model of $\bC_2[\Gamma]\to \M_2[\Gamma]$,
where $\Gamma\subseteq\Gamma_2$ is a fine level.
\begin{enumerate}

\item[(A)] The  Hodge
  conjecture holds for $X$.
\item[(B)] When $\Gamma=\tilde \Gamma(n) $ is a classical level,  the   Tate conjecture hold for
$X$. 
\end{enumerate}

\end{thm}

We deduce this with the help of the following lemmas.

\begin{lemma}\label{lemma:btate}
Let $X_1$ and $X_2$ be smooth projective varieties defined over a finitely
generated field.
\begin{enumerate}
\item If $X_1$ and $X_2$ are birational, then the Tate conjecture holds in  degree $2$ for $X_1$ if and only if it
holds for $X_2$.
\item  If   Tate's conjecture holds in  degree $2$ (respectively $2d$) for $X_1$, and there is a dominant rational  map (respectively surjective regular map) $X_1\dashrightarrow X_2$,
  then the Tate conjecture holds in  degree $2$ (respectively $2d$)
  for $X_2$.
\item If the Tate conjecture holds in degree $2$ for $X_i$,
then the Tate conjecture holds in degree $2$ for $X_1\times X_2$.
\end{enumerate}
  
\end{lemma}

\begin{proof}
  See \cite[thm 5.2]{tate}.
\end{proof}

\begin{lemma}\label{lemma:key}
  Let $f:(X,E)\to (Y,D)$ be a semistable map of 
  smooth projective varieties with    $\dim Y=3$ and $\dim X=4$.
  Suppose that 
$$Gr_F^2H^3(U, R^1f_*\C)=0$$
where  $U= Y-D$.
Then the Hodge conjecture holds for $X$.

\end{lemma}

\begin{proof}
Also let  $\tilde U= X-E$.
Since $X$ is a fourfold,  it is enough to prove that Hodge cycles in $H^4(X)$ are algebraic.
The other cases follow from the Lefschetz $(1,1)$ and hard Lefschetz
theorems. Using  the main theorems of \cite{deligneL, arapura}, and the semisimplicity of
the category of polarizable Hodge structures,
we have a noncanonical isomorphism of Hodge structures
\begin{equation}
  \label{eq:GrWH4tU}
 Gr^W_4 H^4(\tilde U)\cong \underbrace{Gr^W_4 H^4(U,
  f_*\Q)}_{I}\oplus  
\underbrace{Gr^W_4 H^3(U, R^1f_*\Q)}_{II}\oplus \underbrace{Gr^W_4
  H^2(U,R^2 f_*\Q)}_{III}
\end{equation}
The first summand  $I$ can  be identified with
\begin{equation*}
  \begin{split}
\im [H^4(Y)\to H^4(U)] &\cong \frac{ H^4(Y)}{\sum \im H^2(D_i)(-1)}
\\
&\cong L\left(\frac{  H^2(Y)}{\sum L^{-1}\im H^2(D_i)(-1)}\right)
  \end{split}
\end{equation*}
where $L$ is the Lefschetz operator with respect to  an ample divisor
on $Y$. The Lefschetz $(1,1)$ theorem shows that the Hodge cycles in
$I$ are algebraic. 

We have an isomorphism $\Q_U\cong R^2f^o_*\Q$, under which $1\in
H^0(U,\Q)$ maps to the class of a multisection $[\sigma]\in H^0(U, R^2f^o_*\Q) $.
Thus  
the summand $III$ can be identified with
\begin{equation*}
[\sigma]\cup \im [H^2(Y)\to H^2(U)] \cong \frac{[\sigma]\cup H^2(Y)}{\sum
  [\sigma]\cup [D_i]}
\end{equation*}
 It follows again, by the Lefschetz $(1,1)$ theorem, that any Hodge cycle
in the summand  III   is  algebraic.  This is also vacuously
true for II because,
by assumption, there are no Hodge cycles
in $Gr^W_4 H^3(U, R^1f_*\Q)$.

From the sequence
$$\bigoplus H^2(E_i)(-1)\to H^4(X)\to Gr^W_4 H^4(\tilde U)\to 0$$
 we deduce that
$$H^4(X)\cong \bigoplus \im H^2(E_i)(-1)\oplus   Gr^W_4 H^4(\tilde U)
\quad \text{(noncanonically)}$$
Therefore all the Hodge cycles in $H^4(X)$ are algebraic. 
\end{proof}

\begin{lemma}\label{lemma:key2}
  Let $f:(X,E)\to (Y,D)$ be a semistable map of 
  smooth projective varieties defined over a finitely generated
  subfield $K\subset\C$  with    $\dim Y=3$ and $\dim
  X=4$. Let $U= Y-D$.
  Suppose that 
$$Gr_F^2H^3(U, R^1f_*\C)=0,$$
that $H^{1,1}(Y)$ is spanned by algebraic
  cycles, and that the Tate conjecture holds in degree $2$
for the components  $E_i$ of $E$. Then Tate's conjecture holds for $X$
in degree $4$.
\end{lemma}

\begin{proof} 
By the Hodge index theorem 
$$\langle \alpha,\beta\rangle =\pm tr(\alpha\cup \beta)$$
gives a positive definite pairing on the primitive part of $H^4(X)$,
and this can be extended to the whole of $H^4$ by hard Lefschetz. 
Let
\begin{equation*}
  \begin{split}
    S_B &=\sum\im  H^2(E_i(\C),\C)(1) \subseteq H^4(X(\C),\C)(2)\\
S_{Hdg}& =\sum\im  H^1(E_i,\Omega_{E_i}^1)\subseteq
H^2(X,\Omega_X^2)\\
S_\ell &=\sum\im
H_{\text{\'et}}^2(E_i\otimes \bar K, \Q_\ell(1))\subseteq
H_{\text{\'et}}^4(X\otimes \bar K,\Q_\ell(2))
  \end{split}
\end{equation*}
where the images above are with respect to the Gysin maps. Set
\begin{equation*}
  \begin{split}
V_B&= H^4(X(\C),\C)(2)/S_B \\
V_{Hdg} &= H^2(X,\Omega_X^2)/S_{Hdg}\\
V_\ell&= H_{\text{\'et}}^4(X\otimes \bar K,\Q_\ell(2))/S_\ell
\end{split}
\end{equation*}
Observe that $V_B$ is a Hodge structure and $V_\ell$ is a Galois module.
Let us say that a class in any one of these spaces is algebraic if it
lifts to an algebraic cycle in $H^4(X)$ or $H^2(X,\Omega_X^2)$.
Let us write 
$$Tate(-) = \sum_{[L:K]<\infty} (-)^{Gal(\bar K/L)}$$
where $(-)$ can stand for $V_\ell$ or any other Galois module.
Clearly
\begin{equation}
  \label{eq:HT1}
\dim (\text{space of algebraic classes in } V_\ell)\le  \dim Tate(V_\ell)  
\end{equation}

We also claim that
\begin{equation}
  \label{eq:HT}
\dim Tate(V_\ell)\le \dim V_{Hdg}  
\end{equation}
This  will follow from the  Hodge-Tate decomposition. After passing to a finite
extension, we can assume that  all elements of
$Tate(H^4(X,\Q_\ell(2))$ and $Tate(V_\ell)$ are fixed by $Gal(\bar
K/K)$. Let $K_\ell$ denote the completion of $K$ at a prime
lying over $\ell$, and let $\C_\ell = \widehat{\overline{K_\ell}}$. By Faltings \cite{faltingsP} 
there is a  Hodge-Tate decomposition, i.e. a functorial isomorphism of
$Gal(\bar K_\ell/K_\ell)$-modules
$$H_{\text{\'et}}^4(X\otimes \bar K,\Q_\ell(2))\otimes_{\Q_\ell}
\C_\ell\cong \bigoplus_{a+b=4}
H^a(X,\Omega_X^b)\otimes_{K} \C_\ell(2-b)$$
This is  compatible with products, Poincar\'e/Serre  duality, and
cycle maps. Since we can decompose $H_{\text{\'et}}^4(X\otimes \bar K,\Q_\ell(2))= S_\ell\oplus S_\ell^\perp$
as an orthogonal direct sum,  and this is a decomposition of $Gal(\bar
K/K)$-modules, an element of $\gamma\in
Tate(V_\ell)$ can be
lifted to $\gamma_1\in Tate(H_{\text{\'et}}^4(X\otimes \bar K,\Q_\ell(2)))$. This gives a $Gal(\bar
K_\ell/K_\ell)$-invariant element
of $H_{\text{\'et}}^4(X\otimes \bar K,\Q_\ell(2))\otimes \C_\ell$, and thus an element of
$\gamma_2\in H^2(X,\Omega_X^2)\otimes K_\ell$. Let $\gamma_3\in
V_{Hdg}\otimes K_\ell$ denote the image. One can the check that $\gamma\mapsto
\gamma_3$ is a well defined injection of $Tate(V_\ell)\otimes K_\ell\to
V_{Hdg}\otimes K_\ell$.
  This proves  that \eqref{eq:HT} holds.

As in the proof of lemma \ref{lemma:key}, 
we can split
$$V_B(-2)= I\oplus II\oplus III$$
where the summands are defined as in \eqref{eq:GrWH4tU}.
Arguing as above, but with stronger
assumption that $H^{1,1}(Y)$ is algebraic, we can see that the (not
necessarily rational) $(2,2)$
classes in $I$ and $III$ are algebraic, and that $II$ has no such classes.
Therefore $V_{Hdg}$ is spanned by algebraic
classes.  Combined with inequalities \eqref{eq:HT1} and \eqref{eq:HT},
we find that every element of $Tate(V_\ell)$ is an algebraic class.
Therefore given a Tate cycle $\gamma\in Tate(H_{\text{\'et}}^4(X\otimes \bar K,\Q_\ell(2)))$
 there is an algebraic cycle $\gamma'$ so that $\gamma-\gamma'\in
S_\ell$. This means that $\gamma-\gamma'$ is the sum of images of  Tate cycles
in $H^2(E_i)$. By assumption, this is again algebraic.
\end{proof}

\begin{lemma}\label{lemma:picmaximpliestate}
Let $X$ be a smooth   projective variety defined over a finitely generated
  subfield $K\subset\C$. If $H^{1,1}(X)$ is spanned by divisors,  
Tate's conjecture holds for $X$ in degree $2$.
\end{lemma}

\begin{proof}
  This is similar to the previous proof.  We have inequalities
$$\rank NS(X)\le \dim Tate( H_{\text{\'et}}^2(X\otimes \bar
K,\Q_\ell(1)))\le h^{1,1}(X)  $$
where the second follows from  Hodge-Tate.
Since $H^{1,1}(X)$ is spanned by divisors,  we must have equality above.
\end{proof}

\begin{proof}[Proof of theorem \ref{thm:HT}]
  
The statement (A) about the  Hodge conjecture  follows immediately from
corollary~\ref{cor:van} and the lemma \ref{lemma:key}.

We now turn to part  (B) on the Tate conjecture. 
We break the analyis into cases.
Tate in degree 2 follows from theorem~\ref{thm:H2} and
lemma~\ref{lemma:picmaximpliestate}. Hard Lefschetz then implies
Tate in degree 6.
In degree 4, by lemmas \ref{lemma:btate} and \ref{lemma:key2}, it is enough to
 verify that $H^{1,1}(Y)$ is spanned by divisors and that the Tate
 conjecture holds in degree $2$ for  varieties rationally dominating
 components of the divisor $E$. The first condition for  $Y$ is due to
 Weissauer \cite[p. 101]{weissauer}. By corollary \ref{cor:classifyE}, irreducible components of $E$ are dominated by
 $\bC_{1,1}[n]\times \M_{1,1}[n]$, $\bC_{1,2}[n]$, $\bC_{1,1}[m]\times \PP^1$, $\M_{1,1}[m]\times \M_{1,1}[m]$, 
 $\PP^2$ times a curve, or $\PP^3$.
 The Tate conjecture in degree $2$ is trivially true for the last two cases.
The Tate conjecture in degree $2$ for the other cases follows from \cite[thm 5]{gordon} and lemma~\ref{lemma:btate}.
\end{proof}

Part (B) of the previous theorem can be extended slightly. Suppose
that $\Gamma\subseteq \Gamma_2$ is the preimage of a
finite index subgroup of $Sp_4(\Z)$ such that $\M_2[\Gamma]$ is
smooth. With this assumption,  we may 
choose a good model $X\to Y$ of $\bC_2[\Gamma]\to \M_2[\Gamma]$, with $Y=\M_2[\Gamma]$.

\begin{cor}
  Tate's conjecture holds for $X$ as above.
\end{cor}
\begin{proof}
We first note that   $\Gamma$ contains some $\tilde \Gamma(n)$, because the congruence subgroup problem
has a positive solution for $Sp_4(\Z)$ \cite{bms}. Therefore the good
model $X[n]$ for  $\tilde \Gamma(n)$ surjects onto $X$. Since we know
that Tate holds for $X[n]$, it holds for $X$ by lemma \ref{lemma:btate}.
\end{proof}


\end{document}